\documentclass[Final,10pt,reqno]{amsart}
\usepackage{geometry}
\usepackage{amssymb}
\usepackage{graphicx}
\usepackage[all,cmtip]{xy}
\usepackage{amsmath}
\usepackage{dsfont}
\usepackage{mathtools}
\usepackage{extarrows}
\usepackage{tikz-cd}

\usepackage[colorlinks=true,linkcolor=blue,citecolor={blue}]{hyperref}
\newcommand{\comment}[1]{}
\usepackage{multirow}
\usepackage{tabularx}
\usepackage{mathrsfs}
\newtheorem{thm}{Theorem}[section]
\newtheorem{prop}{Proposition}[section]
\newtheorem{defi}{Definition}[section]
\newtheorem{lem}{Lemma}[section]
\newtheorem{cor}{Corollary}[section]
\newtheorem{ex}{Example}[section]

\newcommand{\ro}{\circ}

\renewcommand{\1}{\langle}
\renewcommand{\2}{\rangle}
\newcommand{\met}{\langle\cdot\,,\cdot\rangle}

\newcommand{\rel}{\mathds{R}}
\newcommand{\intg}{\mathds{Z}}

\newcommand{\n}{\nabla}

\newcommand{\w}{\omega}
\newcommand{\et}{\quad \textnormal{and}\quad}

\newcommand{\id}{\operatorname{Id}}

\newcommand{\el}{\operatorname{L}}

\newcommand{\tr}{\operatorname{tr}}

\newcommand{\A}{\mathbf{A}}

\newcommand{\homo}{\operatorname{Hom}}
\newcommand{\End}{\operatorname{End}}

\newcommand{\hol}{\operatorname{Hol}}

\newcommand{\g}{\mathfrak{g}}

\newcommand{\h}{\mathfrak{h}}
\newcommand{\m}{\mathfrak{m}}

\newcommand{\gl}{\mathfrak{gl}}
\newcommand{\so}{\mathfrak{so}}
\newcommand{\p}{\mathfrak{p}}

\newcommand{\hola}{\mathfrak{hol}}

\newcommand{\Ad}{\operatorname{Ad}}
\newcommand{\ad}{\operatorname{ad}}
\newcommand{\Span}{\operatorname{span}}

\newcommand{\aff}{\operatorname{Aff}}
\newcommand{\GL}{\operatorname{GL}}
\newcommand{\SO}{\operatorname{SO}}
\newcommand{\SL}{\operatorname{SL}}
\newcommand{\T}{\operatorname{T}}

\numberwithin{equation}{section}

   \title{Special Affine Connections on Symmetric Spaces}
    \author{Othmane Dani \and Abdelhak Abouqateb}           
\address{Cadi Ayyad University, Faculty of Sciences and Technologies, Department of Mathematics B.P.549 Gueliz Marrkesh. Morocco\newline	a.abouqateb@uca.ac.ma and othmanedani@gmail.com}

\date{\today}
\begin{document}

\maketitle

	{\bf Abstract.} Let $(G,H,\sigma)$ be a symmetric pair and $\g=\m\oplus\h$ the canonical decomposition of the Lie algebra $\g$ of $G$. We denote by $\n^0$ the canonical affine connection on the symmetric space $G/H$. A torsion-free $G$-invariant affine connection on $G/H$ is called special if it has the same curvature as $\n^0$. A special product on $\m$ is a commutative, associative, and $\Ad(H)$-invariant product. We show that there is a one-to-one correspondence between the set of special affine connections on $G/H$ and the set of special products on $\m$. We introduce a subclass of symmetric pairs called strongly semi-simple for which we prove that the canonical affine connection on $G/H$ is the only special affine connection, and we give many examples. We study a subclass of commutative, associative algebra which allows us to give examples of symmetric spaces with special affine connections. Finally we compute the holonomy Lie algebra of special affine connections. 
	\vskip 0.2cm
	
	
	\vspace{1cm}	
\section{Introduction and main results}
In the literature, symmetric spaces are described in a variety of ways. An affine symmetric space in differential geometry is a connected smooth manifold $M$ endowed with an affine connection $\nabla$ such that for each point $p\in M$ there is an affine transformation $\mathfrak{s}_p\in\aff(M,\nabla)$ which fixes $p$ and reverses every geodesic through $p$. On the other hand, in Lie theoretically terms, a symmetric pair is a triple $(G,H,\sigma)$ with $G$ is a connected Lie group, $H$ a closed subgroup of $G$ and $\sigma$ an involutive automorphism  of $G$ such that $G_\sigma^0\subseteq H\subseteq G_\sigma,$ where $G_\sigma$ is the fixed-point subgroup of $\sigma$ and $G_\sigma^0$ its identity compenent. The homogeneous $G$-space $G/H$ is called the corresponding symmetric space. Symmetric spaces play important roles for various branches of mathematics, namely Riemannian geometry, Lie groups theory, harmonic analysis and so on (see for instance \cite{Bertram,Helg,Kob1,Loos,Post}). It is worth noting that for any connected Lie group $G$, we can associate the natural symmetric pair $(G\times G,G,\sigma)$, where $\sigma(a,b)=(b,a) $, and in this case, the corresponding symmetric space $G\times G/G$ is identified with the $G\times G$-homogeneous space $M:=G$ (where the transitive action of $G\times G$ on $G$  is given by $(a,b)\cdot x:=axb^{-1}$ for $a,b,x\in G$). This is why our objective in this article is to describe algebraically a class of invariant affine connections on symmetric spaces which have been explored in \cite{BenBo6} in the particular case of Lie groups. Therefore, our current work could be seen as a natural sequel to the work present in \cite{BenBo6}, where it is proved that if $\g$ is a semi-simple Lie algebra, then there is no non-trivial Poisson structure on $\g$ (this means that there is no non-trivial commutative, associative, and $\ad(\g)$-invariant product on $\g$). Our idea in this framework is to provide correspondences between some algebraic structures and other geometric ones, which could be very useful for geometric or other algebraic questions. 

Let $(G,H,\sigma)$ be a symmetric pair, the tangent map of $\sigma$ at the identity element (also denoted by $\sigma$) induces a splitting $\g=\m\oplus\h$ of the Lie algebra $\g$ of $G$ with $\h:=\ker(\sigma-\id_\g)$ and $\m:=\ker(\sigma+\id_\g)$. Further, one can easily check that $\h$ is the Lie algebra of $H$ and the following inclusions hold
\begin{equation}\label{1}
\Ad(H)(\m)\subseteq \m,\et [\m,\m]\subseteq\h.
\end{equation}
The splitting $\g=\m\oplus\h$ we just defined above is called the canonical decomposition of $\g$ (with respect to $\sigma$). Moreover, since $G/H$ is a reductive homogeneous $G$-space, we denote by $\n^0$ its canonical affine connection, i.e., the unique torsion-free $G$-invariant affine connection for which the geodesics are determined by the exponential map of $G$ (cf. \cite[Theorem $10.1$]{Nom1}). A \textit{special} affine connection on $G/H$ is a torsion-free $G$-invariant affine connection which has the same curvature as the canonical one. As a direct consequence of Nomizu's theorem on invariant affine connections, we will later see that:
\begin{thm}\label{thm1}
Let $(G,H,\sigma)$ be a symmetric pair, $M$ the corresponding symmetric space and $\g=\m\oplus\h$ the canonical decomposition of the Lie algebra $\g$ of $G$. There exists a one-to-one correspondence between the set of special affine connections on $M$ and the set of special products on $\m$, i.e., commutative, associative, and $\Ad(H)$-invariant products on $\m$.
\end{thm}

Our next result is to give some conditions on the symmetric pair $(G,H,\sigma)$ so that the canonical affine connection is the only special affine connection on $G/H$. To be precise, we introduce the following definition. 

\begin{defi}
The symmetric pair $(G,H,\sigma)$ is called 
\begin{enumerate}
\item simple if the isotropy representation $\ad^\m:\h\rightarrow\End(\m)$ is irreducible.
\item semi-simple if the isotropy representation $\ad^\m:\h\rightarrow\End(\m)$ is completely reducible.
\item strongly semi-simple if there exists a family $(\m_i)_{i=1}^k$ of simple $\h$-submodules of $\m$ such that 
\begin{equation}\label{strongsemisimple}
\m=\m_1\oplus\cdots\oplus\m_k,\et [\m_i,\m_i]\neq\{0\},
\end{equation}
for all $i\in\{1,\ldots,k\}$. Such a family is called a strongly decomposition of $\m$.
\end{enumerate}
\end{defi}

Clearly, simple or strongly semi-simple symmetric pairs are semi-simple; moreover, when the isotropy representation is faithful, then according to \cite[pp. $56$]{Nom1} a simple symmetric pair is strongly semi-simple if and only if $\g$ is a semi-simple Lie algebra.\\

Our main result is the following.

\begin{thm}\label{SimpleSymPairs} Let $(G,H,\sigma)$ be a symmetric pair and $\g=\m\oplus\h$ the canonical decomposition of the Lie algebra $\g$ of $G$. If $(G,H,\sigma)$ is simple (with $\dim\m>1$) or strongly semi-simple, then the trivial product is the only special product on $\m$. 
\end{thm}

Note that this result is not valid for semi-simple symmetric pairs as shown by Example \eqref{ExNonStrongSemiSimple}.\\

The paper is organized as follows. In Section \ref{Section2}, we prove Theorem \ref{thm1} and we show that special affine connections are semi-symmetric (see Proposition \ref{PropSemiSym}). Section \ref{Section3} is devoted to proving Theorem \ref{SimpleSymPairs} and providing some examples. In Section \ref{Section4}, we give other examples of strongly semi-simple symmetric pairs. In Section \ref{Section5} we introduce a particular subclass of commutative, associative algebra, what we called commutative, $0$-associative algebra, which allows us to give examples of symmetric spaces with special affine connections (see Proposition \ref{ExpSpAffCon}). Finally, in Section \ref{Section6} we compute the holonomy Lie algebra of a special affine connection.

Until the end of this paper, $(G,H,\sigma)$ will be a symmetric pair, $M:=G/H$ the corresponding symmetric space, $\g=\m\oplus\h$ the canonical decomposition of the Lie algebra $\g$ of $G$, and $\ad^\m:\h\rightarrow\End(\m)$ the isotropy representation of $\h$ in $\m$.

All vector spaces, algebras, etc. in this paper are finite dimensional and over the field of real numbers $\rel$.  

\vskip 0.5cm	

{\bf ACKNOWLEDGMENT} The authors would like to thank Sa\"id Benayadi, University Lorraine-France, for his helpful discussions and advice concerning this paper.

\section{Special Affine Connections on Symmetric Spaces}\label{Section2}

Before going further into the proof of Theorem \ref{thm1}, let us start with some facts that should be known. First of all, since $M$ is a reductive homogeneous $G$-space, then according to Nomizu's Theorem \cite[Theorem $8.1$]{Nom1} there is a one-to-one correspondence between the set of $G$-invariant affine connections on $M$ and the set of bilinear maps $\alpha:\m\times\m\rightarrow\m$ which are invariant by $\Ad(H)$, i.e.,
\begin{equation}\label{2}
\Ad_h\alpha(u,v)=\alpha\left(\Ad_hu,\Ad_hv\right),
\end{equation}
for $u,v\in\m$ and $h\in H$. If $\n$ is a $G$-invariant affine connection on $M$, then it is obvious that the torsion $T^\n$ and the curvature $R^\n$ tensor fields of $\n$ are also $G$-invariant. Thus, they are completely determined by their value at the origin $o\in M$. Hence, under the identification of $T_oM$ with $\m$, using the second inclusion of $\eqref{1}$ in \cite[Formulas $(9.1)$ and $(9.6)$]{Nom1}, the torsion $T^\n$ and the curvature $R^\n$ of $\n$ can be expressed as follows:
\begin{eqnarray}
T^\n(u,v)&=&\alpha^\n(u,v)-\alpha^\n(v,u);\label{3}\\
R^\n(u,v)w&=&\alpha^\n(u,\alpha^\n(v,w))-\alpha^\n(v,\alpha^\n(u,w))-[[u,v],w],\label{4}
\end{eqnarray}
for $u,v,w\in\m$, where $\alpha^\n:\m\times\m\rightarrow\m$ is the bilinear map associated to $\n$.
In particular, for the canonical affine connection $\n^0$, its associated product on $\m$ is the trivial product $\alpha^0=0$. Hence it is torsion-free and its curvature is given by
\begin{equation*}
R^0(u,v)w=-[[u,v],w]\qquad\forall\, u,v,w\in\m.
\end{equation*}

We can now give the following.
\begin{proof} [\textnormal{\textbf{Proof of Theorem \ref{thm1}}}]
Let $\n$ be a special affine connection on $M$. We define a product $\star:\m\times\m\rightarrow\m$ on $\m$ by
\begin{equation*}
u\star v:=\alpha^\n(u,v),
\end{equation*}
for $u,v\in\m$, where $\alpha^\n:\m\times\m\rightarrow\m$ is the bilinear map associated to $\n$. Clearly that the product $\star$ is $\Ad(H)$-invariant and since $\n$ is torsion-free it is commutative by \eqref{3}. Furthermore, using $\eqref{4}$ and the commutativity of $\star$ we obtain that $\star$ is associative.\\
Conversely, given a commutative, associative and $\Ad(H)$-invariant product $\star$ on $\m$, then define a bilinear map $\alpha:\m\times\m\rightarrow\m$ by
\begin{equation*}
\alpha(u,v):=u\star v.
\end{equation*}
Since $\alpha$ is $\Ad(H)$-invariant, it defines a $G$-invariant affine connection $\n^\alpha$ on $M$. Moreover, since $\star$ is commutative, we obtain by \eqref{3} that $\n^\alpha$ is torsion-free. Furthermore, using the fact that $\star$ is commutative, associative and $\eqref{4}$ we get that $\n^\alpha$ has the same curvature as $\n^0$.
\end{proof}

It is clear that any connected Lie group $G$ can be considered as a symmetric space, where the symmetric pair is $G_0:=G\times G$, $H_0:=\Delta G_0$, $\sigma_0:G_0\to G_0,\,(a,b)\mapsto(b,a)$, and the canonical decomposition of the Lie algebra $\g_0$ of  $G_0$ is $\g_0=\m_0\oplus\h_0$, where
\begin{equation*}
\h_0=\big\{(u,u)\,|\,\,u\in\g\big\},\et\m_0=\big\{(u,-u)\,|\,\,u\in\g\big\}.
\end{equation*}
Moreover, the isotropy representation of $H_0$ in $\m_0$ is equivalent to the adjoint representation of $G$ in $\g$, and the canonical connection $\n^0$ on $G$ is the torsion-free bi-invariant affine connection given by
\begin{equation*}
\n^0_{u^+}v^+=\dfrac{1}{2}[u^+,v^+],
\end{equation*}
where $u^+,v^+$ denote the left invariant vector fields on $G$ assosiated respectively to the vectors $u,v\in\g$. So, a special affine connection on $G$ is a torsion-free bi-invariant affine connection on $G$ which has the same curvature as $\n^0$. On the other hand, a special product on $\m_0$ is equivalent to a Poisson structure on $\g$, i.e., a commutative, associative, and $\ad(\g)$-invariant product on $\g$. Thus by Theorem \ref{thm1} we get the following result obtained in \cite[Theorem 2.1]{BenBo6}.
\begin{cor}\label{Thm1BenBo6}
Let $G$ be a connected Lie group and $\g$ its Lie algebra. There is a one-to-one correspondence between the set of special affine connections on $G$ and the set of Poisson structures on $\g$. 
\end{cor}

In differential geometry there is a notion of semi-symmetric spaces which is a direct generalization of locally symmetric spaces, namely, smooth manifolds endowed with a torsion-free affine connection $\n$ for which the curvature tensor $R^\n$ satisfies
\begin{equation*}
\n_X\n_YR^\n-\n_Y\n_XR^\n-\n_{[X,Y]}R^\n=0,
\end{equation*}
for any vector fields $X,Y$. It is known that the above equation is equivalent (see \cite[Chap 4. Formula (26)]{Post}) to the following one
\begin{equation*}
\left[R^\n(X,Y),R^\n(Z,W)\right]=R^\n\left(R^\n(X,Y)Z,W\right)+R^\n\left(Z,R^\n(X,Y)W\right),
\end{equation*}
for any vector fields $X,Y,Z,W$. Hence, since the curvature tensor $R^0$ of the canonical affine connection $\n^0$ satisfies this condition, we easily obtain the following proposition.

\begin{prop}\label{PropSemiSym}
The smooth manifold $M$ endowed with a special affine connection is semi-symmetric. 
\end{prop}

\section{Simple and Strongly Semi-simple Symmetric Pairs}\label{Section3}

In this section we will give a proof for Theorem \ref{SimpleSymPairs}. We begin by the case for which $(G,H,\sigma)$ is a simple symmetric pair (with $\dim\m>1$), then we pass to the strongly semi-simple case.
First notice that, if $\ad^\m:\h\rightarrow\End(\m)$ is not trivial (which is the case if $(G,H,\sigma)$ is simple with $\dim\m>1$ or strongly semi-simple), then $\mathfrak{i}:=\ker\ad^\m$ is an ideal of $\g$ which is strictly contained in $\h$. Thus we get a faithful representation $\overline{\ad}^\m:\overline{\h}\rightarrow\End(\m)$ of the Lie algebra $\overline{\h}:=\h/\mathfrak{i}$. Further, a product on $\m$ is $\ad(\h)$-invariant if and only if it is $\overline{\ad}^\m\left(\overline{\h}\right)$-invariant. Hence, throughout this section we may assume without loss of generality that the isotropy representation of $\h$ in $\m$ is faithful, i.e., $\ad^\m:\h\rightarrow\End(\m)$ is injective.


\comment{
\begin{ex}\textnormal{
If $\h$ is a semi-simple Lie algebra, then by Weyl's theorem the symmetric pair $(G,H,\sigma)$ is semi-simple.}
\end{ex}
}

\begin{proof}[\textnormal{\textbf{Proof of Theorem \ref{SimpleSymPairs} in the case where $(G,H,\sigma)$ is simple}}]
We start with the following remark: since $\m$ is a simple $\h$-module and $[\h,\m]\subseteq\m$ is an $\h$-submodule of $\m$, then it will be either $\{0\}$ or $\m$. But since $\ad^\m$ is faithful it follows that $\m=[\h,\m]$. Now let $\alpha:\m\times\m\rightarrow\m$ be a special product on $\m$. Define
\begin{equation*}
\mathcal{I}:=\big\{u\in\m\,|\,\alpha_u=0\big\}.
\end{equation*}
For $a\in\h$ and $u\in\mathcal{I}$, using \eqref{2} we have 
\begin{equation*}
\alpha_{[a,u]}=[\ad_{a},\alpha_u]=0.
\end{equation*}
Thus $\mathcal{I}$ is an $\h$-submodule of $\m$ and therefore either $\mathcal{I}=\{0\}$ or $\mathcal{I}=\m$. Suppose by contradiction that $\mathcal{I}=\{0\}$. The product on $\m$ given by $u\star v:=\alpha(u,v)$ for $u,v\in\m$ is a special product and hence it is commutative and associative. So for any $u,v\in\m$ and $n\geq 1$
\begin{eqnarray}
\alpha_u^n(v)&=&\alpha_u\ro\alpha_u\ro\cdots\ro\alpha_u(v)\nonumber\\
&=&u\star u\star\cdots\star u\star v\nonumber\\
&=&\alpha_{u^n}(v).\label{alpuNilp}
\end{eqnarray}
On the other hand for $a\in\h$ and $v\in\m$, we have 
\begin{equation*}
\tr\left(\alpha_{[a,v]}\right)=\tr\left(\left[\ad_{a},\alpha_{v}\right]\right)=0.
\end{equation*}
Since each element of $\m$ is a linear combination of elements of $[\h,\m]$, then by $\eqref{alpuNilp}$ we get 
\begin{equation*}
\tr(\alpha^n_u)=0\qquad  \forall\, u\in\m,\, n\geq 1.
\end{equation*}
Hence $\alpha_u$ is a nilpotent endomorphism of $\m$. Let $\widetilde{\m}$ be the vector space
\begin{equation*}
\widetilde{\m}:=\Big\{\alpha_u\in\End(\m)\,|\,\,u\in\m\Big\}.
\end{equation*}
Clearly that $\widetilde{\m}$ is a Lie subalgebra of $\End(\m)$ because $[\alpha_u,\alpha_v]=0$ for $u,v\in\m$. Moreover, each element of $\widetilde{\m}$ is a nilpotent endomorphism of $\m$. Thus by Engel's Theorem there exists a nonzero element $u_0\in\m$
such that
\begin{equation*}
\alpha_{u_0}(u)=\alpha_u(u_0)=0, \quad\forall\, u\in\m.
\end{equation*}
So $\alpha_{u_0}=0$ which implies that $u_0\in\mathcal{I}$. But this constitutes a contradiction and therefore proves the claim.
\end{proof}

\begin{cor}
If $(G,H,\sigma)$ is simple (with $\dim M>1$), then the canonical affine connection $\n^0$ is the only special affine connection on $M$.
\end{cor}

\begin{ex}\textnormal{
If $\g$ is a simple Lie algebra and $H$ is compact, then the symmetric pair $(G,H,\sigma)$ is simple (cf. \cite[Chap 11. Proposition 7.4]{Kob1}). Moreover, since $\g$ is simple we have that $\dim\m>1$ (see \cite[pp. 56]{Nom1}), and therefore the canonical affine connection $\n^0$ is the only special affine connection on $M$.
}\end{ex}

\begin{ex}\textnormal{It is clear that the symmetric pair $(\SO(n+1),\SO(n),\sigma_{J_n})$ is simple ($n>1$), where $\sigma_{J_n}(A):=J_nAJ_n$, with $J_n:=\begin{psmallmatrix}I_n&0\\0&-1 \end{psmallmatrix}$. Thus, the canonical affine connection $\n^0$ is the only special affine connection on the unit $n$-sphere $\mathbb{S}^n$.
}\end{ex}

To demonstrate Theorem \ref{SimpleSymPairs} in the case where $(G,H,\sigma)$ is strongly semi-simple we need the following lemma.

\begin{lem}
If $(G,H,\sigma)$ is strongly semi-simple and $(\m_i)_{i=1}^k$ a strongly decomposition of $\m$, then for each $i,j,l\in\{1,\ldots,k\}$ such that $i\neq j$ and $i\neq l$, we have
\begin{enumerate}
\item[$1.$] $[\m_i,[\m_j,\m_l]]=\{0\}$.
\item[$2.$] $[\m_i,[\m_i,\m_i]]=\m_i$.
\item[$3.$] $Z_i[\m_i,\m_i]:=\Big\{u_i\in\m_i\,|\,\,[u_i,w_i]=0,\,\,\forall\, w_i\in [\m_i,\m_i]\Big\}=\{0\}$.
\end{enumerate}
\end{lem}

\begin{proof}
First we have by Jacobi identity $[\m_i,[\m_j,\m_l]]\subseteq[[\m_i,\m_j],\m_l]+[\m_j,[\m_i,\m_l]]$. Thus
\begin{equation*}
[\m_i,[\m_j,\m_l]]\subseteq \m_i\cap(\m_l\oplus\m_j)=\{0\}.
\end{equation*}
For the second one, using again Jacobi identity we get
\begin{eqnarray*}
[\h,[\m_i,[\m_i,\m_i]]]&\subseteq&[[\h,\m_i],[\m_i,\m_i]]+[\m_i,[\h,[\m_i,\m_i]]]\\
&\subseteq&[\m_i,[\m_i,\m_i]]+[\m_i,[[\h,\m_i],\m_i]]+[\m_i,[\m_i,[\h,\m_i]]]\\
&\subseteq&[\m_i,[\m_i,\m_i]].
\end{eqnarray*}
So $[\m_i,[\m_i,\m_i]]$ is an $\h$-submodule of $\m_i$, and therefore either $[\m_i,[\m_i,\m_i]]=\{0\}$ or $[\m_i,[\m_i,\m_i]]=\m_i$. Suppose by contradiction that $[\m_i,[\m_i,\m_i]]=\{0\}$. Then by the first assertion we obtain that $[[\m_i,\m_i],\m]=\{0\}$, and it follows that $[\m_i,\m_i]=\{0\}$. But this contradicts $\eqref{strongsemisimple}$ and hence $[\m_i,[\m_i,\m_i]]=\m_i$. For the last one a similar argument shows that $Z_i[\m_i,\m_i]$ is an $\h$-submodule of $\m_i$ and therefore $Z_i[\m_i,\m_i]=\{0\}$, because otherwise we would have $[\m_i,[\m_i,\m_i]]=\{0\}$, which contradicts the second assertion.
\end{proof}

\begin{proof}[\textnormal{\textbf{Proof of Theorem \ref{SimpleSymPairs} in the strongly semi-simple case}}] The proof is very similar to the proof of the case where $(G,H,\sigma)$ is simple. Let $(\m_i)_{i=1}^k$ be a strongly decomposition of $\m$ and $\alpha:\m\times\m\rightarrow\m$ a special product on $\m$. Define
\begin{equation*}
\mathcal{I}:=\big\{u\in\m\,|\,\alpha_u=0\big\}.
\end{equation*}
Clearly that $\mathcal{I}$ is an $\h$-submodule of $\m$, and so our task is to proving that $\mathcal{I}=\m$. Before doing this, we will show that the product $\alpha$ respects the strongly decomposition of $\m$ i.e., for $i,j\in\{1,\ldots,k\}$ such that $i\neq j$ we have
\begin{equation*}
\alpha(\m_i,\m_i)\subseteq\m_i,\et \alpha(\m_i,\m_j)=\{0\}.
\end{equation*}
For $u_i,v_i\in\m_i$, $\alpha(u_i,v_i)$ can be written uniquely in the form
\begin{equation*}
\alpha(u_i,v_i)=\alpha(u_i,v_i)_1+\cdots+\alpha(u_i,v_i)_k.
\end{equation*}
So for each $j\neq i$ and $w_j\in[\m_j,\m_j]\subset \h$, using the $\Ad(H)$-invariance of $\alpha$ and the first assertion in the previous lemma, we get
\begin{eqnarray*}
[\alpha(u_i,v_i)_j,\w_j]&=&[\alpha(u_i,v_i),\w_j]\\
&=&\alpha([u_i,w_j],v_i)+\alpha(u_i,[v_i,w_j])\\
&=&0.
\end{eqnarray*}
Thus $\alpha(u_i,v_i)_j\in Z_j[\m_j,\m_j]=\{0\}$ and therefore $\alpha(\m_i,\m_i)\subseteq\m_i$. A similar argument shows that $\alpha(u_i,u_j)\in\m_i\oplus\m_j$ for $u_i\in\m_i$ and $u_j\in\m_j$. Moreover, if we write $u_l=[v_l,w_l]$, for $v_l\in\m_l,w_l\in[\m_l,\m_l]$ and $l=i,j$, we obtain 
\begin{eqnarray*}
\alpha(u_i,u_j)&=&[\alpha(v_i,u_j),w_i]-\alpha(v_i,[u_j,w_i])\\
&=&[\alpha(v_i,u_j),w_i]\\
&=&[\alpha(v_i,u_j)_i,w_i]\in\m_i\,,
\end{eqnarray*}
and similarly
\begin{eqnarray*}
\alpha(u_i,u_j)&=&[\alpha(u_i,v_j),w_j]-\alpha([u_i,w_j],v_j)\\
&=&[\alpha(u_i,v_j),w_j]\\
&=&[\alpha(u_i,v_j)_j,w_j]\in\m_j.
\end{eqnarray*}
Hence $\alpha(u_i,u_j)\in\m_i\cap\m_j=\{0\}$, which implies that $\alpha(\m_i,\m_j)=\{0\}$. Now suppose by contradiction that $\mathcal{I}\neq \m$. Since $\mathcal{I}$ is an $\h$-submodule of $\m$, then by changing the indexation of the sequence $(\m_i)_{i=1}^k$ we can assume that there exists $1\leq r\leq k$ such that for $i\in\{1,\ldots,r-1\}$ and $j\in\{r,\ldots,k\}$ we have
\begin{equation*}
\mathcal{I}\cap\m_i=\m_i,\et \mathcal{I}\cap\m_j=\{0\}.
\end{equation*}
Thus $\m=\mathcal{I}+\mathcal{J}$ with $\mathcal{J}:=\m_r\oplus\cdots\oplus\m_k$. Indeed this is a direct sum, to see it let $u\in \mathcal{I}\cap \mathcal{J}$, then write $u=u_r+\cdots+u_k$ for $u_j\in\m_j$ and $j\in\{r,\ldots,k\}$. For $w_j\in[\m_j,\m_j]$ we have
\begin{equation*}
[u_j,w_j]=[u,w_j]\in \mathcal{I}\cap\m_j.
\end{equation*}
Hence $u_j=0$ and it follows that $u=0$. This implies in turn that $\m=\mathcal{I}\oplus\mathcal{J}$. On the other hand, we denote by $\widetilde{\alpha}$ the restriction of $\alpha$ to $\mathcal{J}$, then the product on $\mathcal{J}$ given by $u\star v:=\widetilde{\alpha}(u,v)$ for $u,v\in\mathcal{J}$ is a special product and hence it is commutative and associative. So for any $u,v\in\mathcal{J}$ and $n\geq 1$ 
\begin{eqnarray}
\widetilde{\alpha}_u^n(v)&=&\widetilde{\alpha}_u\ro\widetilde{\alpha}_u\ro\cdots\ro\widetilde{\alpha}_u(v)\nonumber\\
&=&u\star u\star\cdots\star u\star v\nonumber\\
&=&\widetilde{\alpha}_{u^n}(v).\nonumber
\end{eqnarray}
Furthermore, every element $u$ of $\m$ can be expressed as a linear combination of elements of the form $[v_i,w_i]$ for $v_i\in\m_i,w_i\in[\m_i,\m_i]$ and $i\in\{1,\ldots,k\}$. Then using this and the $\Ad(H)$-invariance of $\alpha$ we can easily deduce that $\tr(\alpha_u)=0$. Thus for $u\in\mathcal{J},n\geq 1$ one has $\tr(\widetilde{\alpha}_u^n)=0$, and therefore $\widetilde{\alpha}_u$ is a nilpotent endomorphism of $\mathcal{J}$. Let $\widetilde{\mathcal{J}}$ be the vector space
\begin{equation*}
\widetilde{\mathcal{J}}:=\Big\{\widetilde{\alpha}_u\in\End(\mathcal{J})\,|\,\,u\in\mathcal{J}\Big\}.
\end{equation*}
It is obvious that $\widetilde{\mathcal{J}}$ is a Lie subalgebra of $\End(\mathcal{J})$. Furthermore, each element of $\widetilde{\mathcal{J}}$ is a nilpotent endomorphism of $\mathcal{J}$. Thus by Engel's Theorem there exists $u_0\in\mathcal{J}\backslash\{0\}$
such that
\begin{equation*}
\widetilde{\alpha}_{u_0}(u)=\widetilde{\alpha}_u(u_0)=0, \quad\forall u\in\mathcal{J}.
\end{equation*}
Hence $\widetilde{\alpha}_{u_0}=0$. But since the restriction of $\alpha_{u_0}$ to $\mathcal{I}$ vanishes, we deduce that $\alpha_{u_0}=0$ and then $u_0\in\mathcal{I}$. This constitues a contradiction and proves the claim.
\end{proof}

\begin{cor}
If $(G,H,\sigma)$ is strongly semi-simple, then the canonical affine connection $\n^0$ is the only special affine connection on $M$.
\end{cor}

\begin{ex}\textnormal{Let $G$ be a connected semi-simple Lie group, $\g$ its Lie algebra and $(G_0,H_0,\sigma_0)$ its associated symmetric pair. Since the isotropy representation of $\h_0\cong\g$ in $\m_0\cong\g$ is equivalent to the adjoint representation of $\g$ and $\g$ is semi-simple, there exists a family $(\g_i)_{i=1}^k$ of simple ideals of $\g$ such that
\begin{equation*}
\g=\g_1\oplus\cdots\oplus\g_k,\et [\g_i,\g_i]=\g_i,\qquad\forall\,i\in\{1,\ldots,k\}.
\end{equation*}
Hence the symmetric pair $(G_0,H_0,\sigma_0)$ is strongly semi-simple. 
}\end{ex}

Consequently, we obtain the following corollary.
\begin{cor}
On a semi-simple connected Lie group $G$, the canonical affine connection $\n^0$ is the only special affine connection. 
\end{cor}

Although the conclusion of this corollary fails if we replace the semi-simplicity by the reductivity as the next proposition shows.

\begin{prop}
Every reductive non semi-simple connected Lie group $G$ admits a special affine connection which is different from the canonical one.
\end{prop}

\begin{proof}
Let $e_0$ be a nonzero element in the center of the Lie algebra $\g$ of $G$. Define a product $\star:\g\times\g\to\g$ on $\g$ as follows:
\begin{equation*}
u\star v:=\kappa_{\g}(u,v)e_0,
\end{equation*} 
where $\kappa_{\g}:\g\times\g\to\rel$ is the Killing form of $\g$. A straightforward computation shows that $\star$ is a non-trivial Poisson product on $\g$. Hence the result follows by using Corollary \ref{Thm1BenBo6}.
\end{proof}

\begin{ex}\textnormal{
Let $(G_i,H_i,\sigma_i),\,i=1,2$ be two simple symmetric pairs and $\g_i=\m_i\oplus\h_i$ their corresponding canonical decompositions. We assume that $\ad^{\m_i}:\h_i\rightarrow\End(\m_i)$ are faithfuls and $\g_1,\g_2$ are semi-simple Lie algebras. It is clear that $(G^\times:=G_1\times G_2,H^\times:=H_1\times H_2,\sigma^\times:=\sigma_1\times\sigma_2)$ is a symmetric pair and the corresponding canonical decomposition is $\g^\times=\m^\times\oplus\h^\times$, where
\begin{equation*}
\m^\times=\m_1\oplus\m_2,\et\h^\times=\h_1\oplus\h_2.
\end{equation*}
Since the adjoint representation of $\h_i$ in $\m_i$ is irreducible, we get that $\m_i$ is a simple $\h^\times$-submodule of $\m^\times$. In addition, using the fact that $\g_i$ is semi-simple we obtain (see \cite[pp. 56]{Nom1}) that $[\m_i,\m_i]=\h_i$. Hence, it follows that $(G^\times,H^\times,\sigma^\times)$ is strongly semi-simple.
}\end{ex}

Now the question naturally arises whether an analogous statement for semi-simple symmetric pairs remains true. The answer to this question is no in general as the next example shows.

\begin{ex}\label{ExNonStrongSemiSimple}\textnormal{ Let $H$ be the Lie group given by
\begin{equation*}
H:=\left\{\begin{pmatrix}
A&0\\
0&1
\end{pmatrix}\,|\,\,A\in \SO(3)\right\}.
\end{equation*}
Consider the Lie group $G:=\rel^4\rtimes H$ and define an involutive automorphism $\sigma$ of $G$ by:
\begin{equation*}
\sigma:G\rightarrow G,\quad (x,\widetilde{A})\mapsto (-x,\widetilde{A}),
\end{equation*}
for $x\in\rel^4$ and $\widetilde{A}:=\begin{psmallmatrix}
A&0\\
0&1
\end{psmallmatrix}\in H$. It is easy to check that $(G,H,\sigma)$ is a symmetric pair. Moreover, the canonical decomposition of the Lie algebra $\g$ of $G$ is $\g=\m\oplus_{\rtimes}\h$, where
\begin{equation*}
\m=\big\{(u,0)\in\g\,|\,\,u\in\rel^4\big\},
\end{equation*}
and 
\begin{equation*}
\h=\left\{(0,\widehat{X})\in\g\,|\,\,\widehat{X}:=\begin{pmatrix}
X&0\\
0&0
\end{pmatrix},\,X\in\so(3)\right\}.
\end{equation*}
On the other hand, since the Lie bracket of $\g$ is given by
\begin{equation*}
\big[(u,\widehat{X}),(v,\widehat{Y})\big]:=\left(\widehat{X}v-\widehat{Y}u,\widehat{[X,Y]}\right),\qquad\forall\,\,(u,\widehat{X}),(v,\widehat{Y})\in\g.
\end{equation*}
Then, under the identification of $\m$ with $\rel^4$, we obtain that the isotropy representation of $\h$ in $\m$ is
\begin{equation*}
\begin{array}{ccccc}  \ad^\m & : &\h & \longrightarrow & \gl(4,\rel)
\\& & (0,\widehat{X}) & \longmapsto & \widehat{X}. 
\end{array}
\end{equation*}
Let $(e_i)_{1\leq i\leq 4}$ be the canonical basis of $\rel^4$, then we can easily check that $\m_0:=\Span\{e_4\}$ and $\m_1:=\Span\{e_1,e_2,e_3\}$ are simple $\h$-submodules of $\m$ such that $\m=\m_0\oplus\m_1$. Thus the isotropy representation of $\h$ in $\m$ is completely reducible and therefore $(G,H,\sigma)$ is semi-simple. But $(G,H,\sigma)$ is not strongly semi-simple, because $[\m,\m]=\{0\}$.}

\textnormal{
Now, we will show that there exists a non-trivial special product on $\m\cong\rel^4$. First we identify $\rel^4$ with $\rel^3\times\rel$ and we denote by $\met$ the Euclidean inner product on $\rel^{3}$. Define the following product
\begin{equation*}
\star:\rel^4\times\rel^4\rightarrow\rel^4,\quad\textnormal{written}\quad (x_1,t_1)\star(x_2,t_2):=\left(0,\1x_1,x_2\2\right). 
\end{equation*}
It is obvious that $\star$ is a non-trivial commutative, associative product on $\rel^4$. Moreover, for $x:=(x_1,t_1),y:=(x_2,t_2)\in\rel^4$ and $\widetilde{A}\in H$, we have
\begin{eqnarray*}
\widetilde{A}x\star\widetilde{A}y&=&(Ax_1,t_1)\star (Ax_2,t_2)\\
&=&\left(0,\1Ax_1,Ax_2\2\right)\\
&=&\left(0,\1x_1,x_2\2\right)\\
&=&\widetilde{A}(x\star y).
\end{eqnarray*}
Thus the product $\star$ is $\Ad(H)$-invariant and hence it is a non-trivial special product on $\m$.
}\end{ex}

\section{Examples of Strongly Semi-simple Symmetric Pairs}\label{Section4}

This section is devoted to give some examples of strongly semi-simple symmetric pairs, namely, Cartan's symmetric pairs and Semi-simple Riemannian symmetric pairs. Before going further, we start by recalling some definitions and properties that will be needed later.
In what follows, $(\g,[\,,])$ will be a real Lie algebra and $\kappa_\g:\g\times\g\rightarrow\rel$ its Killing form.  

\begin{defi}
A Cartan involution of $\g$ is an involutive automorphism $\tau$ of $\g$ such that the symmetric bilinear form $\met:\g\times\g\rightarrow\rel$ written $\1u,v\2:=-\kappa_\g(u,\tau(v))$, is positive definite.
\end{defi}

Note that if $\tau$ is a Cartan involution of $\g$, then $\g$ splits as a direct sum of $\h^\tau:=\ker(\tau-\id_{\g})$ and $\m^\tau:=\ker(\tau+\id_{\g})$. Moreover, since $\met$ is positive definite we get that the Killing form $\kappa_\g$ of $\g$ is negative definite on $\h^\tau$ and positive definite on $\m^\tau$. Further, $\met$ is $\ad(\h^\tau)$-invariant and the following inclusions hold
\begin{equation*}
[\h^\tau,\h^\tau]\subseteq \h^\tau,\quad [\h^\tau,\m^\tau]\subseteq\m^\tau,\et [\m^\tau,\m^\tau]\subseteq \h^\tau.
\end{equation*}
The decomposition $\g=\m^\tau\oplus \h^\tau$ is called the \textit{Cartan decomposition} with respect to $\tau$, and the inclusion $[\h^\tau,\m^\tau]\subseteq\m^\tau$ gives rise to a representation $\ad^{\m^\tau}:\h^\tau\rightarrow\End(\m^\tau)$ which also called the isotropy representation of $\h^\tau$ in $\m^\tau$. Note that the fact that $\met$ is positive definite implies that $\g$ is a semi-simple Lie algebra and it is compact if and only if $\tau=\id_\g$.

\begin{prop}\label{CartanProperties}
Let $\tau$ be a Cartan involution of $\g$ and $\g=\m^\tau\oplus\h^\tau$ the corresponding Cartan decomposition. Then 
\begin{enumerate}
\item[$1.$] If $(\g_i)_{i=1}^k$ is a family of simple ideals of $\g$ such that $\g=\bigoplus_{i=1}^k\g_i$, then for $i\neq j$, $\g_i$ and $\g_j$ are mutually orthogonal with respect to $\kappa_\g$.
\item[$2.$] $\h^\tau$ and $\m^\tau$ are mutually orthogonal with respect to $\kappa_\g$.
\item[$3.$] If $\p$ is a nonzero $\h^\tau$-submodule of $\m^\tau$, then $[\m^\tau,\p]=[\p,\p]$. In particular, $[\p,\p]\neq\{0\}$ and $\p\oplus[\p,\p]$ is an ideal of $\g$. 
\end{enumerate}
\end{prop}

\begin{proof}
The first and the second statement are clear. For the last one, let $\p$ be a nonzero $\h^\tau$-submodule of $\m^\tau$ and denote by $\p^\perp\subset\m^\tau$ its orthogonal complement with respect to $\met$, i.e $\m^\tau=\p\oplus\p^\perp$. Take $u\in\p$ and $v\in\p^\perp$, then we have
\begin{eqnarray*}
\1[u,v],[u,v]\2&=&\kappa_\g\big([v,u],[u,v]\big)\\
&=&\kappa_\g\big(v,[u,[u,v]]\big)\\
&=&\1[[u,v],u],v\2\\
&=&0.
\end{eqnarray*}
Thus $[\p,\p^\perp]=\{0\}$ and therefore $[\m^\tau,\p]=[\p,\p]$. If $[\p,\p]=\{0\}$, then $\p$ will be a nonzero abelian ideal of $\g$, which is impossible because $\g$ is semi-simple. Finally, using Jacobi identity we can easily check that $\p\oplus[\p,\p]$ is an ideal of $\g$.
\end{proof}

\begin{defi} 
The symmetric pair $(G,H,\sigma)$ is called a Cartan symmetric pair if the tangent map of $\sigma$ at the identity element (also denoted by $\sigma$) is a Cartan involution of the Lie algebra $\g$ of $G$.
\end{defi} 

\begin{ex}\textnormal{
The example type is $(\SL(n,\rel),\SO(n),\sigma^*)$, where $\sigma^*$ is given by $\sigma^*(A):=\left(A^{-1}\right)^{\T}$. Geometrically, the symmetric space associated to this symmetric pair is the set of all real symmetric positive definite $n$-matrices with determinant $1$.
}\end{ex}

The following proposition shows that all Cartan's symmetric pairs are strongly semi-simple.

\begin{prop}
If $(G,H,\sigma)$ is a Cartan symmetric pair, then it is strongly semi-simple.
\end{prop}

\begin{proof}
First, let us show that the isotropy representation of $\h$ in $\m$ is completely reducible. To do this it suffices to prove that each $\h$-submodule of $\m$ possesses an $\h$-submodule complement. Let $\p\subseteq\m$ be an $\h$-submodule of $\m$ and denote by $\p^\perp\subseteq\m$ its orthogonal complement with respect to $\met$. Clearly that $\m=\p\oplus\p^\perp$, and for $u\in\p^\perp,v\in\p,a\in\h$ we have
\begin{equation*}
\1[a,u],v\2=-\1u,[a,v]\2=0.
\end{equation*} 
Hence $\p^\perp$ is an $\h$-submodule of $\m$. Now, if $(\m_i)_{i=1}^k$ is a family of simple $\h$-submodules of $\m$ such that $\m=\bigoplus_{i=1}^k\m_i$, then using the last assersion in Proposition \ref{CartanProperties} we deduce that  $[\m_{i},\m_{i}]\neq\{0\}$ for all $i\in\{1,\ldots,k\}$. Thus $(\m_i)_{i=1}^k$ is a strongly decomposition of $\m$.
\end{proof}

Now, recall that, if $\g$ is a simple Lie algebra and $H$ is compact, then $(G,H,\sigma)$ is a simple symmetric pair with $\dim \m>1$. It follows that the canonical affine connection is the only special affine connection on $M$. The following proposition shows that the last conclusion remains true if we replace the simplicity of $\g$ by the semi-simplicity.

\begin{prop}
If $\g$ is a semi-simple Lie algebra and $H$ is compact, then the symmetric pair $(G,H,\sigma)$ is strongly semi-simple.
\end{prop}

\begin{proof}
Since $H$ is compact, let $\met$ be an $\Ad(H)$-invariant inner product on $\g$, and define a linear endomorphism $\phi:\m\to\m$ by
\begin{equation*}
\kappa_\g(u,v)=\1\phi(u),v\2,\qquad\forall\,u,v\in\m.
\end{equation*}
A direct computation using the fact that $\met$ and $\kappa_\g$ are both $\Ad(H)$-invariant, symmetric bilinear forms, one can easily check that $\phi$ is symmetric with respect to $\met$ and commutes with all $\ad_u$ for $u\in\h$. Thus, there is a direct sum decomposition $\m=\bigoplus_{i=1}^r\p_i$ such that $\phi_{|\p_i}=t_i\id_{\p_i}$ with $t_i\in\rel^*$ and $t_i\neq t_j$ for $i\neq j$. Moreover, $(\p_i)_{i=1}^r$ are $\h$-submodules of $\m$ which are mutually orthogonal with respect to $\met$. Hence there exists a direct sum decomposition $\m=\bigoplus_{i=1}^k\m_i$ such that each $\m_i$ is a simple $\h$-submodule of $\m$ which is contained in some $\p_{i'}$ and $(\m_i)_{i=1}^k$ are mutually orthogonal with respect to $\met$. Furthermore, if $u_i\in\m_i$ and $u_j\in\m_j$, then
\begin{eqnarray*}
\kappa_\g([u_i,u_j],[u_i,u_j])&=&\kappa_\g(u_j,[[u_i,u_j],u_i])\\
&=&t_{j'}\1u_j,[[u_i,u_j],u_i]\2\\
&=&0.
\end{eqnarray*}
Since $\ad_u:\g\rightarrow\g$ is skew-symmetric with respect to $\met$ for $u\in\h$ and $\g$ is semi-simple, one can easily check that $\kappa_\g$ is negative definite on $\h$ and therefore $[\m_i,\m_j]=\{0\}$. Thus $[\m_i,\m_i]\neq \{0\}$ for all $i\in\{1,\ldots,k\}$, which proves that $(G,H,\sigma)$ is strongly semi-simple.
\end{proof}



\section{Examples of Symmetric Spaces with Special Affine Connections}\label{Section5}
This section is devoted to give examples of symmetric spaces on which there is a special affine connection which is different from the canonical one. We start by recalling some basic facts on how one can get a symmetric space from a Jordan algebra.

\begin{defi} A Jordan algebra is a commutative algebra $(\A,.)$ in which the following identity 
\begin{equation*}
x.(x^2.y)=x^2.(x.y),
\end{equation*}
holds.
\end{defi}

\begin{ex}
The trivial example is a commutative, associative algebra.
\end{ex}

It is well known (see for example \cite{Jacob}) that to each Jordan algebra $(\A,.)$, we can associate (Tits-Kantor-Koecher construction) a $\intg_2$-grading of a Lie algebra $\g^\A=\h^\A\oplus\m^\A$ as follows: we define
\begin{equation*}
\g^\A_{_{-1}}:=\A,\qquad\g^\A_{_{0}}:=\Span\big\{\el_x,\left[\el_y,\el_z\right]\,/\,x,y,z\in\A\big\}\subset\End(\A),
\end{equation*}
and 
\begin{equation*}
\g^\A_{_{1}}:=\Span\big\{\el,\left[\el_x,\el\right]\,/\,x\in\A\big\}\subset\homo(S^2\A,\A),
\end{equation*}
where $\el(x,y)=\el_xy:=x.y$, and $\left[\el_x,\el\right](y,z):=\left[\el_x,\el_y\right](z)-\el_{x.y}(z)$, for $x,y,z\in\A$. Then we set
\begin{equation*}
\g^\A:=\g^\A_{_{0}}\oplus\g^\A_{_{-1}}\oplus\g^\A_{_{1}},
\end{equation*}
and we get that $\g^\A$ is a short $\intg$-grading of a Lie algebra with the following Lie bracket:
\begin{enumerate}
\item[•] $[x,y]=[A,B]:=0$, for $x,y\in\g^\A_{_{-1}}$ and $A,B\in\g^\A_{_{1}}$;
\item[•] $[F,x]:=F(x)$, for $x\in\g^\A_{_{-1}}$ and $F\in\g^\A_{_{0}}$;
\item[•] $[F,B](x,y):=F\big(B(x,y)\big)-B\big(F(x),y\big)-B\big(x,F(y)\big)$, for $x,y\in\g^\A_{_{-1}}$, $F\in\g^\A_{_{0}}$ and $B\in\g^\A_{_{1}}$;
\item[•] $[B,x](y):=B(x,y)$, for $x\in\g^\A_{_{-1}}$ and $B\in\g^\A_{_{1}}$.
\end{enumerate}

Hence, if we set $\h^\A:=\g^\A_{_{0}}$ and $\m^\A:=\g^\A_{_{-1}}\oplus\g^\A_{_{1}}$, we deduce that $\g^\A=\h^\A\oplus\m^\A$ is a $\intg_2$-grading of a Lie algebra, i.e.,
\begin{equation*}
\left[\h^\A,\h^\A\right]\subseteq\h^\A,\quad\left[\h^\A,\m^\A\right]\subseteq\m^\A,\et\left[\m^\A,\m^\A\right]\subseteq\h^\A.
\end{equation*}
In summary, any Jordan algebra $(\A,.)$ gives rise to a $\intg_2$-grading of a Lie algebra $\g^\A=\h^\A\oplus\m^\A$, and therefore (see \cite[Theorem I$.1.3$]{Bertram}) to a simply connected symmetric space $M^\A$.\\

Now, we introduce a particular subclass of associative algebras, which will be used to construct our examples.

\begin{defi}
An associative algebra $(\A,.)$ is called $0$-associative if 
\begin{equation*}
x.y.z=0,\qquad\forall\,x,y,z\in\A.
\end{equation*}
\end{defi}
\begin{ex}
Let $(V,+)$ be an $n$-dimensional vector space, and $(e_i)_{1\leq i\leq n}$ any basis of it. For $i_1,i_2\in\left\{1,\ldots,n\right\}$ fixed such that $i_1\neq i_2$, the product given by
\begin{equation*}
e_{i_1}.e_{i_1}=e_{i_2},\et e_{i}.e_{j}=0,
\end{equation*}
for $\{i,j\}\neq\{i_1,i_1\}$, is (commutative) $0$-associative.
\end{ex}

\begin{ex} Let $(\A,.)$ be a symmetric Leibniz algebra, i.e., an algebra $(\A,.)$ such that for any $x,y\in\A$, we have
\begin{equation*}
\left[\el_x,\el_y\right]=\el_{x.y},\et \left[\operatorname{R}_x,\operatorname{R}_y\right]=\operatorname{R}_{y.x},
\end{equation*}
where $\el_x,\operatorname{R}_x\in\End(\A)$ are defined by $\el_x(y):=x.y$ and $\operatorname{R}_x(y):=y.x$. If we consider the product $\ast$ on $\A$ given by
\begin{equation*}
x\ast y:=x.y+y.x,\qquad\forall\,x,y\in\A,
\end{equation*}
then a small computation shows that $(\A,\ast)$ is a (commutative) $0$-associative algebra.
\end{ex}

The proof of the following proposition is a matter of pure computation and is thus omitted.

\begin{prop} Let $(\A,.)$ be a commutative, associative algebra, then $\h^\A$ is an abelian Lie subalgebra of $\g^\A$. Moreover, if $(\A,.)$ is $0$-associative, then $\g^\A$ is a $2$-step nilpotent algebra.
\end{prop}

Now, we can give a way of getting examples of symmetric spaces on which there is a special affine connection which is different from the canonical one. More precisely, we have:

\begin{prop}\label{ExpSpAffCon}
Let $(\A,.)$ be a commutative, $0$-associative algebra, then there is a special affine connection on its associated symmetric space $M^\A$ which is different from the canonical one.
\end{prop}

\begin{proof}
According to Theorem \eqref{thm1} it suffices to define a commutative, associative, and $\ad(\h^\A)$-invariant product on $\m^\A$. We consider the following product on $\m^\A$ given by
\begin{equation*}
\star:\m^\A\times\m^\A\to\m^\A,\quad (x+A)\star(y+B):=x.y,
\end{equation*}
where $"."$ is the commutative, $0$-associative product of $\A$. One can easily check that the product $\star$ is commutative, associative, and $\ad(\h^\A)$-invariant. 
\end{proof}

\section{Holonomy Lie Algebra of Special Affine Connections}\label{Section6}
In this last section, we compute the holonomy Lie algebra of a special affine connection. But first, let us start with some background that should be known. 

Given an affine connection $\n$ on $M$, for any loop $\gamma$ at $p\in M$ the parallel transport along $\gamma$ is a linear isomorphism of $T_pM$, and the set of such linear isomorphisms for all loops at $p$ forms a group which is called the \textit{holonomy group} of $\n$ based at $p$ and denoted by $\hol_p(\n)$. The \textit{restricted holonomy group} $\hol^0_p(\n)$ is the subgroup composed of parallel transports along all contractible loops at $p$. It is well known (see \cite[Chap 2. Theorem 4.2]{Kob1}) that $\hol^0_p(\n)$ is the identity component of $\hol_p(\n)$ and that $\hol^0_p(\n)$ is a connected Lie group. The \textit{holonomy Lie algebra} of $\n$ based at $p$ is the Lie algebra of $\hol^0_p(\n)$. On the other hand, consider the vector subspace $\hola_p^\n$ of $\End(T_pM)$ which is generated by all linear endomorphisms of the form $R^\n(u,v),(\n_wR^\n)(u,v),(\n_z\n_wR^\n)(u,v),\ldots$, where $u,v,w,z,\ldots$ are arbitrary tangent vectors at $p$. It was shown in \cite[Lemma 4.2]{Nij} that it is a Lie subalgebra of $\End(T_pM)$ and we call it the \textit{infinitesimal holonomy Lie algebra} at $p$. The immersed Lie subgroup of $\GL(T_pM)$ generated by $\hola_p^\n$ is the \textit{infinitesimal holonomy group} at $p$. The main result (see \cite[Theorem 7]{Nij}) is that the restricted holonomy group is equal to the infinitesimal holonomy group at every point.

According to our discussion above, since the curvature tensor $R^0$ of the canonical affine connection $\n^0$ is parallel (i.e., $\n^0R^0=0$), then under the identification of $\m$ with $T_oM$, the holonomy Lie algebra of $\n^0$ at the origin $o\in M$ is given by
\begin{equation*}
\hola^{\n^0}_o=\ad_{[\m,\m]}.
\end{equation*}  
If $\n$ is an arbitrary $G$-invariant affine connection on $M$, then by \cite[Chap 10. Theorem 4.4]{Kob1} and under the identification of $\m$ with $T_oM$, the holonomy Lie algebra of $\n$ at $o\in M$ is the smallest Lie subalgebra $\hola^\n_o$ of $\End(\m)$ which satisfies the following two conditions
\begin{enumerate}
\item\label{1LemHolLiAlg} For all $u,v\in\m,\, R^\n(u,v)\in\hola^\n_o$;
\item\label{2LemHolLiAlg} For all $u\in\m,\, \left[\alpha^\n_u,\hola^\n_o\right]\subseteq\hola^\n_o$,
\end{enumerate} 
where $\alpha^\n:\m\times\m\to\m$ is the bilinear map associated to $\n$ and $\alpha^\n_u\in\End(\m)$ is defined by $\alpha^\n_u(v):=\alpha^\n(u,v)$. Although the holonomy Lie algebra of a $G$-invariant affine connection on $M$ is difficult to compute explicitly, it turns out that the holonomy Lie algebra of a special affine connection on $M$ can be easily computed as the next proposition shows. 

\begin{prop}
Let $\n$ be a special affine connection on $M$. Then the holonomy Lie algebra of $\n$ at the origin $o\in M$ is given by
\begin{equation*}
\hola^\n_{o}=\ad_{[\m,\m]}+\alpha^\n_{[[\m,\m],\m]},
\end{equation*}
where $\alpha^\n:\m\times\m\rightarrow\m$ is the bilinear map associated to $\n$.
\end{prop}
\begin{proof}
For a special affine connection $\n$ on $M$, by \eqref{2} and \eqref{4} one has
\begin{equation*}
[\ad_{[u,v]},\alpha^\n_w]=\alpha^\n_{[[u,v],w]},\quad R^\n(u,v)=-\ad_{[u,v]},\et [\alpha^\n_u,\alpha^\n_v]=0,
\end{equation*}
for all $u,v,w\in\m$. Thus the Lie algebra $\ad_{[\m,\m]}+\alpha^\n_{[[\m,\m],\m]}$ satisfies \ref{1LemHolLiAlg} and \ref{2LemHolLiAlg}, so it contains $\hola_{o}^\n$. On the other hand, for $x,y,z,u,v\in\m$
\begin{eqnarray*}
\ad_{[u,v]}+\alpha^\n_{[[x,y],z]}&=&\ad_{[u,v]}+\left[\ad_{[x,y]},\alpha^\n_z\right]\\
&=&R^\n(v,u)+\left[\alpha^\n_z,R^\n(x,y)\right]\in\hola_{o}^\n.
\end{eqnarray*}
This proves the other inclusion, and hence the claim.
\end{proof}

\vskip 0.5cm

{\bf ACKNOWLEDGMENT}
We thank the referee for carefully reading the paper, as well as for his valuable suggestions and helpful corrections which significantly improved the article.

\bibliographystyle{amsplain}

\end{document}